\documentclass[12pt]{article}
\usepackage{amsmath,amssymb,amsthm}
\usepackage{hyperref}
\newcommand{\tmmathbf}[1]{\ensuremath{\boldsymbol{#1}}}

\newtheorem{theorem}{Theorem}[section]
\newtheorem{corollary}[theorem]{Corollary}
\newtheorem{lemma}[theorem]{Lemma}

\newtheorem{definition}[theorem]{Definition}
\newtheorem{remark}[theorem]{Remark}

\numberwithin{equation}{section}
\usepackage{amssymb,latexsym,amsmath,epsfig,amssymb,amsthm, amsmath,color}

\makeatletter
\definecolor{nicegreen}{rgb}{0,0.5,0}
\definecolor{green}{rgb}{0,0.5,0.5}
\definecolor{grape}{rgb}{0.8,0,0.5}

\def\F{\mathbb{F}_q}
\def\P{\mathcal{P}}
\def\L{\mathcal{L}}

\def\S{\mathcal{S}}
\def\V{\mathcal{V}}
\def\Q{\mathcal{Q}}

\def\HH{\mathcal{H}}

\def\v{\mathbf{v}}
\def\p{\mathbf{p}}
\def\q{\mathbf{q}}

\def\FF{\mathbb{F}_q}
\def\E{\mathcal{E}}
\begin{document}
\title{Incidence bounds and applications over finite fields}
\author{Nguyen Duy Phuong \thanks{ Vietnam National University,
    Email: {\tt duyphuong@vnu.edu.vn
}}\and
  Thang Pham\thanks{EPFL, Lausanne, Switzerland. Research partially supported by Swiss National Science Foundation
Grants 200020-144531 and 200021-137574.
    Email: {\tt thang.pham@epfl.ch}}
\and
Nguyen Minh Sang\thanks{ Vietnam National University,
    Email: {\tt sangnmkhtnhn@gmail.com
}}
\and
Claudiu Valculescu\thanks{ EPFL, Lausanne, Switzerland. Research partially supported by Swiss National Science Foundation
Grants 200020-144531 and 200021-137574.
    Email: {\tt adrian.valculescu@epfl.ch}}
\and
    Le Anh Vinh\thanks{Vietnam National University, Email: {\tt vinhla@vnu.edu.vn
}}
}
\date{}
\maketitle
\begin{abstract}
In this paper we introduce a unified approach to deal with incidence problems between points and varieties over finite fields. More precisely, we prove that the number of incidences $I(\P, \V)$ between a set $\P$ of points and a set $\V$ of varieties of a certain form satisfies
\[\left\vert I(\P,\V)-\frac{|\P||\V|}{q^k}\right\vert\le q^{dk/2}\sqrt{|\P||\V|}.\]

This result is a generalization of the results of Vinh (2011), Bennett et al. (2014), and Cilleruelo et al. (2015). As applications of our incidence bounds, we obtain results on the pinned value problem and the Beck type theorem for points and spheres.

Using the approach introduced, we also obtain a result on the number of distinct distances between points and lines in $\mathbb{F}_q^2$, which is the finite field analogous of a recent result of Sharir et al. (2015).
\end{abstract}
\section{Introduction}
In 1983, Szemer\'{e}di and Trotter \cite{st} proved that for any set $\P$ of $n$ points, and any set $\L$ of $n$ lines in the plane, the number of incidences between points of $\P$ and lines from $\L$ is asymptotically at most $n^{4/3}$. Apart from being interesting in itself and being a useful tool for various other discrete mathematics problem, the Szemer\'{e}di-Trotter theorem allowed various exentsions and generalizations. T\'{o}th proved that the same bound holds when we work over the complex plane (see \cite{to} for more details). Pach and Sharir \cite{ps} generalized the Szemer\'{e}di-Trotter theorem to the case of points and curves \cite{ps}.

Let $\F$ be a finite field of $q$ elements where $q$ is an odd prime power. Let $\P$ be a set of points and $\mathcal{L}$ be a set of lines in $\mathbb{F}_q^2$, and $I(\P,\L)$ be the number of incidences between $\P$ and $\L$.  In \cite{bourgain-katz-tao}, Bourgain, Katz, and Tao proved that if one has $N$ lines and $N$ points in the plane $\F^2$ for some $1\ll N\ll q^2$, then there are at most $O(N^{3/2-\epsilon})$ incidences. Here and throughout, $X\gtrsim Y$ means that $X\ge CY$ for some constant $C$ and $X\gg Y$ means that $Y=o(X)$, where $X,Y$ are viewed as functions of the parameter $q$.   The study of incidence problems over finite fields received a considerable amount of attention in recent years \cite{iosevichetal, gro,hr, jo, kollar, thang, solymosi, lund,tv, vinh-incidence, vinhpointline}.

Note that the bound $N^{3/2}$ can be easily obtained from extremal graph theory.  The relation between $\epsilon$ and $\alpha$ in the result of Bourgain, Katz, and Tao is difficult to determine, and it is far from being tight. If $N=\log_2\log_6\log_{18} q-1$, then Grosu \cite{gro} proved that one can embed the point set and the line set to $\mathbb{C}^2$ without changing the incidence structure. Thus it follows from a tight bound on the number of incidences between points and lines in $\mathbb{C}^2$ due to T\'{o}th \cite{to} that $I(\P,\L)=O(N^{4/3})$. By using methods from spectral graph theory, the fifth listed author \cite{vinh-incidence} proved the tight bound for the case $N\gg q$ as follows.
\begin{theorem}[\textbf{Vinh}, \cite{vinh-incidence}]\label{vinh-pl}
Let $\P$ be a set of points and $\L$ be a set of lines in $\mathbb{F}_q^2$. Then we have
\begin{equation}\label{vinh-in}\left\vert I(\P,\L)-\frac{|\P||\L|}{q}\right\vert\le q^{1/2}\sqrt{|\P||\L|}.\end{equation}
\end{theorem}

It follows from Theorem \ref{vinh-pl} that when $N\ge q^{3/2}$, the number of incidences between $\P$ and $\L$ is asymptotically at most $(1+o(1))N^{4/3}$ (this meets the Szemer\'{e}di-Trotter bound). Furthermore, if $|\P||\L|\gg q^3$, then the number of incidences is close to the expected value $|\P||\L|/q$. The lower bound in the theorem is also proved to be sharp up to a constant factor, in the sense that there is a set of points $\P$ and a set of lines $\L$ with $|\P|=|\L|=q^{3/2}$ that determines no incidence (for details see \cite{vinhpointline}).
Theorem \ref{vinh-pl} has various applications in several combinatorial number theory problems (for example \cite{hen1, hen2, area}).

The main purpose of this paper is to introduce a unified approach, which allows us to deal with incidence problems between points and certain families of varieties. As applications of incidence bounds, we obtain results on the pinned value problem and the Beck type theorem for points and spheres. Using this approach, we also obtain a result on the number of distinct distances between points and lines in $\mathbb{F}_q^2$.

\subsection{Incidences between points and varieties}
We first need the following definitions.
\begin{definition}
Let $S$ be a set of polynomials in $\F[x_1,\ldots,x_d]$. The variety determined by $S$ is defined as follows
\[V(S):=\{\mathbf{p}\in \F^d ~\colon f(\mathbf{p})=0 ~\textup{for all }~ f\in S\}.\]
\end{definition}

Let $h_1(\mathbf{x}), \ldots, h_k(\mathbf{x})$ be fixed polynomials of degree at most $q-1$ in $\mathbb{F}_q[x_1,\ldots,x_d]$, and let $\mathbf{b}_i=(b_{i1},\ldots,b_{id})$, with $1\leq i\leq k$, be fixed vectors in $(\mathbb{Z}^+)^d$, and $\gcd(b_{ij}, q-1)=1$ for all $1\le j\le d$. For any $k$-tuple $(\mathbf{a}_1,\ldots, \mathbf{a}_k)$ with $\mathbf{a}_i=(a_{i1},\ldots,a_{id}, a_{i(d+1)})\in \F^{d+1}$, we define
\[f_i(\mathbf{x}, \mathbf{a}_i):=h_i(\mathbf{x})+\mathbf{a}_i\cdot \mathbf{x}^{\mathbf{b}_i},~\textup{where}~ \mathbf{a}_i\cdot \mathbf{x}^{\mathbf{b}_i}:=\sum_{j=1}^da_{ij}x_j^{b_{ij}}+a_{i(d+1)}.\]

Also, we define the corresponding families of varieties as follows:
\[\label{hai}V_{\mathbf{a}_1,\ldots,\mathbf{a}_k}:=V\left(x_{d+1}-f_1\left(\mathbf{x},\mathbf{a}_1),\ldots,x_{d+k}-f_k(\mathbf{x},\mathbf{a}_k\right)\right)\subseteq \F^{d+k}, \text{ and}\]
\[\label{mot}W_{\mathbf{a}_1,\ldots,\mathbf{a}_k}:=V\left(f_1(\mathbf{x},\mathbf{a}_1),\ldots,f_k(\mathbf{x},\mathbf{a}_k)\right) \subseteq \F^{d}.\]

Similarly to incidences between points of lines, given a set $\P$ of points and a set $\V$ of varieties, we define the number of incidences  $I(\P,\V)$ between $\P$ and $\V$ as the cardinality of the set $\{(p,v) \in \P \times \V \mid p \in v\}$. Our first main result is as follows:
\begin{theorem}\label{dl2}
Let $\P$ be a set of points in $\mathbb{F}_q^d\times \mathbb{F}_q^k$ and $\V$ a set of varieties of the form $V_{\mathbf{a}_1,\ldots, \mathbf{a}_k}$ defined above. Then the number of incidences between $\P$ and $\V$ satisfies
\[\left\vert I(\P,\V)-\frac{|\P||\V|}{q^k}\right\vert\le q^{dk/2}\sqrt{|\P||\V|}.\]
\end{theorem}
As a consequence of Theorem \ref{dl2}, we obtain the following result.
\begin{corollary}\label{dl1}
Let $\P$ be a set of points in $\mathbb{F}_q^d$ and $\V$ be a set of varieties of the form $W_{\mathbf{a}_1,\ldots, \mathbf{a}_k}$ defined above. Then the number of incidences between $\P$ and $\V$ satisfies
\[\left\vert I(\P,\V)-\frac{|\P||\V|}{q^k}\right\vert \le q^{dk/2}\sqrt{|\P||\V|}.\]
\end{corollary}
Let us observe that if $h_i(\mathbf{x})\equiv 0$ and $\tmmathbf{b}_i=(1,\ldots,1)$ for all $1\le i\le k$, then a variety of the form $V_{\mathbf{a}_1,\ldots,\mathbf{a}_k}$ is a $k$-flat in the vector space $\F^{d+k}$. Therefore, we recover the bound established by Bennett et al. \cite{iom} on the number of incidences between points and flats:
\begin{corollary}[\textbf{Bennett et al.} \cite{iom}]\label{benben}
Let $\P$ be a set of points, and $\mathcal{F}$ a set of $k$-flats in $\F^{d+k}$. Then the number of incidences between $\P$ and $\mathcal{F}$ satisfies
\[\left\vert I(\P,\mathcal{F})-\frac{|\P||\mathcal{F}|}{q^k}\right\vert \le q^{dk/2}\sqrt{|\P||\mathcal{F}|}.\]
\end{corollary}
It follows from Theorem \ref{dl2} and Theorem \ref{dl1} that if $|\P||\V|\ge 2q^{k(d+2)}$, then $\P$ and $\V$ determine at least one incidence. Also if $|\P||\V|\gg 2q^{k(d+2)}$, then the number of incidences is close to the expected value $|\P||\V|/q^k$.

There are some applications of Corollary \ref{benben} in combinatorial geometry problems, for instance, the number of congruent classes of triangles determined by a set of points in $\mathbb{F}_q^2$ in \cite{iom}, and the number of right angles determined by a point set in $\mathbb{F}_q^d$ in \cite{gabor}.

When $k=1$, varieties of the form $W_{\mathbf{a}_1,\ldots,\mathbf{a}_k}$ become hypersurfaces in $\mathbb{F}_q^d$, so they can be written as
\begin{equation}W_{\mathbf{a}}=V\left(h(\mathbf{x})+a_1x_1^{b_1}+\cdots+a_dx_d^{b_d}+a_{d+1}\right),~\mathbf{a}=(a_1,\ldots, a_d)\in \F^d.\end{equation}
Therefore, we obtain the following bound on the number of incidences between points and hypersurfaces:
\begin{theorem}\label{hello11}
Let $\P$ be a set of points in $\F^d$, and $\S$ a set of hypersurfaces of the form $W_{\mathbf{a}}$. Then the number of incidences between $\P$ and $\S$ satisfies
\[\left\vert I(\P,\S)-\frac{|\P||\S|}{q}\right\vert \le q^{d/2}\sqrt{|\P||\S|}.\]
\end{theorem}
When $h(\mathbf{x})=x_1^2+\cdots+x_d^2$, $\tmmathbf{a} = (1,\ldots,1)$ and $b_1 = \ldots =b_d = 2$, as a consequence of Theorem \ref{hello11}, we recover the bound on the number of incidences between points and spheres obtained in \cite{iosevichetal, thang}.
\begin{corollary}[\textbf{Cilleruelo et al.} \cite{iosevichetal}]\label{hello12}
Let $\P$ be a set of points, and $\S$ a set of spheres with arbitrary radii in $\F^d$. Then the number of incidences between $\P$ and $\S$ satisfies
\[\left\vert I(\P,\S)-\frac{|\P||\S|}{q}\right\vert \le q^{d/2}\sqrt{|\P||\S|}.\]
\end{corollary}
Theorem \ref{hello12} also has various applications in several combinatorial problems over finite fields, for instance, Erd\H{o}s distinct distance problem, the Beck type theorem for points and circles, and subset without repeated distance, see \cite{iosevichetal, thang} for more details.
\subsection{Pinned values and Distinct radii}
\paragraph{Pinned values problem:} The distance function between two points $\mathbf{x}$ and $\mathbf{y}$ in $\F^d$, denoted by $||\mathbf{x}-\mathbf{y}||$, is defined as $||\mathbf{x}-\mathbf{y}||=(x_1-y_1)^2+\cdots+(x_d-y_d)^2$.
Although it is not a norm, the function $||\mathbf{x}-\mathbf{y}||$ has properties similar to the Euclidean norm (for example, it is invariant under orthogonal matrices).

Bourgain, Katz, and Tao \cite{bourgain-katz-tao} were the first to consider the the finite analogue of the classical Erd\H{o}s distinct distance problem, namely to determine the smallest possible cardinality of the set
$\Delta_{\FF} (\E)=\{||\mathbf{x}-\mathbf{y}||=(x_1-y_1)^2+\cdots+(x_d-y_d)^2\colon \mathbf{x}, \mathbf{y}\in \E\}\subset \mathbb{F}_q,$
where $\E\subset\mathbb{F}_q^d$. More precisely, they proved that $|\Delta_{\FF} (\E)|\gtrsim |\E|^{1/2+\epsilon}$, where $|\E|=q^\alpha$ and $\epsilon>0$ is a small constant depending on $\alpha$.

Iosevich and Rudnev \cite{iosevich} studied the following question: how large does $\E\subset\mathbb{F}_q^d$, $d\ge 2$ have to be, so that $\Delta_{\mathbb{F}_q}(\E)$ contains a positive proportion of the elements of $\mathbb{F}_q$. They proved that if $\E\subset \mathbb{F}_q^d$ such that $|\E|\gtrsim Cq^{d/2}$ for sufficiently large $C$, then $|\Delta_{\FF}(\E)|=\Omega\left(\min\left\lbrace q, q^{-(d-1)/2}|\E|\right\rbrace\right)$ (in other words, for any sufficiently large $\E\subseteq\mathbb{F}_q^d$, the set $\Delta_{\mathbb{F}_q}(\E)$ contains a positive proportion of the elements of $\mathbb{F}_q$). From this, one obtains that that if $|\E| \gtrsim q^{(d+1)/2}$, then $|\Delta_{\FF}(\E)|\gtrsim q$. This is in fact directly related to Falconer's result \cite{fa} in Euclidean space, saying that for every set $\E$ with Hausdorff dimension greater that $(d+1)/2$, the distance set is of positive measure.

Hart et al. \cite{hartetal} proved that the exponent $(d+1)/2$ is the best possible in odd dimensions, although in even dimensions, it might still be place for improvement.  Chapman et al. \cite{chapman1} showed that if a set $\E\subseteq \mathbb{F}_q^2$ satisfies $|\E|\ge q^{4/3}$, then $|\Delta_{\FF}(\E)|$ contains a positive proportion of the elements of $\mathbb{F}_q$. In the same paper it was also proved that for any set $\P$ of points in $\mathbb{F}_q^d$ with $|\P|\ge q^{(d+1)/2}$, there exists a subset $\P'$ in $\P$, such that $|\P'|=(1-o(1))|\P|$, and for any $\mathbf{y}\in \P'$, $|\Delta_{\mathbb{F}_q}(\P, \mathbf{y})|\gtrsim q$, where $\Delta_{\F}(\P, \mathbf{y})=\{||\tmmathbf{x}-\tmmathbf{y}||\colon \tmmathbf{x}\in \P\}$. (which is the pinned distance problem)

Let $Q(\mathbf{x})$ be a non-degenerate quadratic form. For a fixed non-square element $\lambda\in \F\setminus\{0\}$,  the quadraic form $Q(\mathbf{x})$ can be written as
 \[Q(\mathbf{x})=x_1^2-x_2^2+x_3^2-x_4^2+\cdots+x_{2m-1}^2-\epsilon x_{2m}^2,~~\text{if}~~ d=2m,\] and
\[Q(\mathbf{x})=x_1^2-x_2^2+x_3^2-x_4^2+\cdots+x_{2m-1}^2-x_{2m}^2+\epsilon x_{2m+1}^2,~~\text{if}~~d=2m+1,\]
 where $\epsilon \in \left\lbrace 1, \lambda \right\rbrace$, see \cite{integral} for more details.

Given a point $\tmmathbf{q}\in \F^d$ and a set of points $\P\subseteq \F^d$, we define the pinned distance set determined by $Q(\mathbf{x})$ and $\q$ as $\Delta_Q(\P,\tmmathbf{q})=\{Q(\tmmathbf{p}-\tmmathbf{q})\colon \tmmathbf{p}\in \P\}$. Using methods from spectral graph theory, the fifth listed author obtained the following:
\begin{theorem}[{\bf Vinh} \cite{vinh-norm}]\label{pinchap}
Let $\P$ be a set of points in $\mathbb{F}_q^d$ such that $|\P|\ge
q^{(d+1)/2}$, then there exists a subset $S \subset \P$ such that $|S|=(1-o(1))|\P|$, and for any $\tmmathbf{y}\in S$, we have  $|\Delta_Q(\P, \tmmathbf{y})|\gtrsim q$.
\end{theorem}

In our paper, as an application of Theorem \ref{dl2} and using a similar approach to the one in \cite{iosevichetal}, we generalize Theorem \ref{pinchap} to \textit{non-degenerate polynomials}.
If $F(\mathbf{x},\mathbf{y})$ is a polynomial in $\mathbb{F}_q[x_1,\ldots,x_d, y_1,\ldots,y_d]$, we say that $F(\mathbf{x}, \mathbf{y})$ is \textit{non-degenerate} if $F(\mathbf{x}, \mathbf{y})$ can be written as
\[F(\mathbf{x},\mathbf{y}):=g(\mathbf{x},\mathbf{y})+(x_1^{b_1},\ldots,x_d^{b_d})M(y_1^{c_1},\ldots,y_d^{c_d})^T,\]
 where $g(\mathbf{x}, \mathbf{y})=g_1(\mathbf{x})+g_2(\mathbf{y})\in \mathbb{F}_q[x_1,\ldots,x_d, y_1,\ldots,y_d]$, $M$ is a $d\times d$ invertible matrix, and $\gcd(c_i,q-1)=1$ for all $1\le i \le d$.
\begin{theorem}\label{hello2}
Let $F(\mathbf{x}, \mathbf{y})$ be a non-degenerate polynomial and $\P$ be a set of points in $\mathbb{F}_q^d$ such that $|\P|\ge
(\sqrt{1-c^2}/c^2)q^{(d+1)/2}$ for some constant $0<c<1$. Then there is $\P' \subset \P$ such that $|\P'|\ge (1-c)|\P|$, and for any $\tmmathbf{y}\in \P'$, $|\Delta_F(\P, \tmmathbf{y})|\ge (1-c) q$, where $\Delta_F(\P,\tmmathbf{q})=\lbrace F(\tmmathbf{p},\tmmathbf{q})\colon \tmmathbf{p}\in \P \rbrace.$
\end{theorem}
\begin{corollary}\label{bodeds}
Let $F(\mathbf{x}, \mathbf{y})$ be a non-degenerate polynomial and $\P$, $\Q$ be sets of points in $\mathbb{F}_q^d$ such that $|\P||\Q|\ge 2\sqrt{3}q^{d+1}$ for some constant $0<c<1$. Then there is $\P' \subset \P$ such that $|\P'|\ge |\P|/2$, and for any $\tmmathbf{y}\in \P'$, $|\Delta_F(\Q, \tmmathbf{y})|\ge  q/2$, where $\Delta_F(\Q,\tmmathbf{q})=\lbrace F(\tmmathbf{p},\tmmathbf{q})\colon \tmmathbf{p}\in \Q \rbrace.$
\end{corollary}
\paragraph{The Beck type theorem for points and spheres:} Let $\P$ be a set of points in $\mathbb{F}_q^2$. Iosevich, Rudnev, and Zhai \cite{area} made the first investigation on the finite fields analogue of the Beck type theorem for points and lines in $\mathbb{F}_q^2$. More precisely, they proved that if $|\P|\ge 64q\log q$, then the number of distinct lines determined by $\P$ is at least $q^2/8$. In \cite{lund},  Lund and Saraf improved the condition of the cardinality of $\P$ to $3q$. Recently, Cilleruelo et al. \cite{iosevichetal} studied the Beck type theorem for points and circles in $\mathbb{F}_q^2$ by employing the lower bound on the number of incidences between points and circles in $\mathbb{F}_q^2$. Formally, their result is as follows.
\begin{theorem}[\textbf{Cilleruelo et al.} \cite{iosevichetal}]\label{becktype}
Let $\P$ be a set of points in $\mathbb{F}_q^2$. If $|\P|\ge 5q$, then the number of distinct circles determined by $\P$ is at least $4q^3/9$.
\end{theorem}
As a consequence of Theorem \ref{becktype}, we obtain the following result.
\begin{theorem}
Let $\P$ be a set of $5q$ points in $\mathbb{F}_q^2$. Then the number of distinct radii of circles determined by $\P$ is at least $4q/9$.
\end{theorem}
Note that it is hard to generalize Theorem \ref{becktype} in higher dimensional cases by their arguments. In the following theorem, we will give an approach to address this problem by using a result on the number of pinned distinct distances.
\begin{theorem}\label{beckthm}
Let $\P$ be a set of $8q^2$ points in $\mathbb{F}_q^3$. Then the number of distinct spheres determined by $\P$ is at least $q^4/9$.
\end{theorem}
As a consequence of Theorem \ref{beckthm}, we obtain the following result on the number of distinct radii of spheres determined by a set of points in $\mathbb{F}_q^3$.
\begin{theorem}\label{radii}
Let $\P$ be a set of $8q^2$ points in $\mathbb{F}_q^3$. Then the number of distinct radii of spheres determined by $\P$ is at least $q/9$.
\end{theorem}
\begin{remark}
We note that one can follow the proof of Theorem \ref{beckthm} to prove that there exist constants $c=c(d)$ and $c'=c'(d)$ such that there are at least $cq^{d+1}$ $d$-dimensional spheres determined by a set of $c'q^{d-1}$ points in $\mathbb{F}_q^d$.
\end{remark}
\subsection{Distinct distances between points and lines}
As already mentioned in the abstract, we use the same approach to address the finite field variants of two recent results due to Sharir et al. \cite{sharir}, involving distances between points and lines. The first bound is a lower bound for the minimum number of distinct distances between a set of points and a set of lines, both in the plane. A second result is a lower bound for the minimum number of distinct distances between a set of non-collinear points and the lines that they span.

\begin{theorem}[{\bf Sharir et al.} \cite{sharir}]
For $m^{1/2}\leq n\leq m^2$, the minimum number $D(m,n)$ of point-line distances
between $m$ points and $n$ lines in $\mathbb{R}^2$ satisfies $D(m,n)=\Omega(m^{1/5}n^{3/5})$
\end{theorem}

\begin{theorem}[{\bf Sharir et al.} \cite{sharir}]
The minimum number $H(m)$ of point-line distances between $m$ non-collinear
points and their spanned lines satisfies $H(m)=\Omega(m^{4/3})$.
\end{theorem}
In the plane over finite fields, a line $ax+by+c=0$ is \textit{degenerate} if and only if $a^2+b^2=0$. Similarly, a hyperplane $a_1x_1+\cdots+a_dx_d+a_{d+1}=0$ in $\mathbb{F}_q^d$ is \textit{degenerate} if and only if $a_1^2+\cdots+a_d^2=0$. For a point $p=(x_p,y_p)\in\F^2$ and a non-degenerate line $l:\;ax+by+c=0$ in $\F^2$, let $d(p, l)$ denote the distance function between $p$ and $l$, defined as \[d(p, l)=\frac{(ax_p+by_p+c)^2}{a^2+b^2}.\] For a set of points $\P$ in $\F^2$ and a line $l$, set $\Delta_{\F}(\P,l)=\{d(p,l): p\in\P\}$. Distances between points and non-degenerate lines are preserved under rotations and translations.

Similarly, for a point $p=(x_p^1,x_p^2,\ldots,x_p^d)\in\F^d$ and a non-dengenerate hyperplane $h:a_1x_1+\cdots+a_dx_d+a_{d+1}=0$, we define the point-hyperplane distance \[d(p, h)=(a_1x_p^1+\cdots+a_dx_p^d+a_{d+1})^2/(a_1^2+\cdots+a_d^2).\] For a set of points $\P$ in $\F^d$ and a hyperplane $h$, we let $\Delta_{\F}(\P,h)=\{d(p,h): p\in\P\}$.

We prove that under a similar condition as in the result due to Chapman et al. \cite{chapman1} on the number of distinct distances between points in $\F^2$, the set of distances between $\P$ and $\L$ contains a positive proportion of the elements of $\mathbb{F}_q$.
\begin{theorem}\label{thm1}
Let $\P$ be a set of points and $\L$ be a set of non-degenerate lines in $\F^2$, such that
\[|\P||\L|\ge \frac{4(1-c^2)}{(1/2-(1-c^2))^2}q^{8/3}\]
with $1-c^2<1/4$. Then there exists a subset $\L'$ of $\L$ with $|\L'|=(1-o(1))|\L|$, so that $|\Delta_{\F}(\P, l)| \gtrsim q$, for each line $l$ in $\L'$.
\end{theorem}
Combining a finite field variant of Beck's theorem (which can be found in \cite{lund}) with Theorem \ref{thm1}, we obtain the following bound on the number of distinct distances between a set of points and their spanned lines.
\begin{corollary}\label{co1}
Let $\P$ be a set of points in $\F^2$ with $|\P|\ge 3q$, and let $\L$ be the set of lines spanned by $\P$ in $\F^2$.
Then there exists a subset $\L'$ of $\L$ with $|\L'|=(1-o(1))|\L|$, so that $|\Delta_{\F}(\P, l)|\gtrsim q$, for each line $l$ in $\L'$.
\end{corollary}

By similar arguments as in the proof of Theorem \ref{thm1}, we obtain a similar result on the number of distinct distances between points and hyperplanes in $d$-dimensional vector space over finite fields as follows.
\begin{theorem}\label{dplthm1}
Let $\P$ be a set of points in $\F^d$, and $\HH$ be a set of non-degenerate hyperplanes in $\F^d$, such that
\[|\P||\HH|\ge \frac{4(1-c^2)}{(1/2-(1-c^2))^2}q^{4d/3},\]
with $1-c^2<1/4$. Then there exists a subset $\HH'$ of $\HH$ with $|\HH'|=(1-o(1))|\HH|$, so that $|\Delta_{\F}(\P, h)|\gtrsim q$, for each line $h$ in $\HH'$.
\end{theorem}
\section{Tools}
This section contains a couple of notions and theorems that we use as tools in the proofs of our main results.
We fist state the well-known Schwartz-Zippel Lemma (for proof refer to Theorem $6.13$ in \cite{li}).
\begin{lemma}[\textbf{Schwartz-Zippel}]\label{sch} Let $P(\tmmathbf{x})$ be a non-zero polynomial of degree $k$. Then
\[\left\vert \{\tmmathbf{x}\in \mathbb{F}_q^d: P(\tmmathbf{x})=0\}\right\vert\le kq^{d-1}.\]
\end{lemma}

We say that a bipartite graph is \emph{biregular} if in both of its two parts, all vertices have the same degree. If $A$ is one of the two parts of a bipartite graph, we write $\deg(A)$ for the common degree of the vertices in $A$.
Label the eigenvalues so that $|\lambda_1|\geq |\lambda_2|\geq \cdots \geq |\lambda_n|$.
Note that in a bipartite graph, we have $\lambda_2 = -\lambda_1$.
The following variant of the expander mixing lemma is proved in \cite{eustis}. We include the proof of this result for the sake of completeness of the paper.
\begin{lemma}[\textbf{Expander mixing lemma}]\label{expander} Let $G$ be a bipartite graph with parts $A, B$ such that the vertices in $A$ all have degree $a$ and the vertices in $B$ all have degree $b$. Then, for any two sets $X\subset A$ and $Y\subset B$, the number of edges between $X$ and $Y$, denoted by $e(X,Y)$, satisfies
\[\left\vert e(X,Y)-\frac{a}{|B|}|X||Y|\right\vert\le \lambda_3\sqrt{|X||Y|},\] where $\lambda_3$ is the third eigenvalue of $G$.
\end{lemma}
\begin{proof}
We assume that the vertices of $G$ are labeled from $1$ to $|A|+|B|$, and we denote by $M$ the adjacency matrix of $G$ having the form
$$
M=\begin{bmatrix}
0&N\\
N^t&0
\end{bmatrix},
$$ where $N$ is the $|A|\times|B|$ $0-1$ matrix, with $N_{ij}=1$ if and only if there is an edge between $i$ and $j$. First, let us recall some properties of the eigenvalues of the matrix $M$. Since all vertices in $A$ have degree $a$ and all vertices in $B$ have degree $b$, all eigenvalues of $M$ are bounded by $\sqrt{ab}$. Indeed, let us denote the $L_1$ vector norm by $||\cdot||_1$, and let $\mathbf{e}_v$ be the unit vector having $1$ in the position corresponding to vertex $v$ and zeroes elsewhere. One can observe that $||M^2\cdot \mathbf{e}_v||_1\le ab$, so the absolute value of each eigenvalue of $M$ is bounded by $\sqrt{ab}$. Let $\textbf{1}_X$ denote the column vector of size $|A|+|B|$ having $1\textrm{s}$ in the positions corresponding to the set of vertices $X$ and $0\textrm{s}$ elsewhere. Then, we have that
\[M(\sqrt{a}\mathbf{1}_A+\sqrt{b}\mathbf{1}_B)=b\sqrt{a}\mathbf{1}_B+a\sqrt{b}\mathbf{1}_A=\sqrt{ab}(\sqrt{a}\mathbf{1}_A+\sqrt{b}\mathbf{1}_B),\]
\[M(\sqrt{a}\mathbf{1}_A-\sqrt{b}\mathbf{1}_B)=b\sqrt{a}\mathbf{1}_B-a\sqrt{b}\mathbf{1}_A=-\sqrt{ab}(\sqrt{a}\mathbf{1}_A-\sqrt{b}\mathbf{1}_B),\]
which implies that $\lambda_1=\sqrt{ab}$ and $\lambda_2=-\sqrt{ab}$ are the first and second eigenvalues, corresponding to the eigenvectors $(\sqrt{a}\mathbf{1}_A+\sqrt{b}\mathbf{1}_B)$ and $(\sqrt{a}\mathbf{1}_A-\sqrt{b}\mathbf{1}_B)$.

Let $W^\perp$ be a subspace spanned by the vectors $\textbf{1}_A$ and $\textbf{1}_B$. Since $M$ is a symmetric matrix, the eigenvectors of $M$, except $\sqrt{a}\textbf{1}_A+\sqrt{b}\textbf{1}_B$ and $\sqrt{a}\textbf{1}_A-\sqrt{b}\textbf{1}_B$, span $W$. Therefore, for any $u\in W$,  $Mu\in W$, and $||Mu||\le \lambda_3||u||$. Let us now remark the following facts:
\begin{itemize}
\item[1.] Let $K$ be a matrix of the form $\begin{bmatrix}
0&J\\
J&0
\end{bmatrix},$ where $J$ is the $|A|\times |B|$ all-ones matrix. If $u\in W$, then $Ku=0$ since every row of $K$ is either $\textbf{1}_A^T$ or $\textbf{1}_B^T$.
\item[2.] If $w\in W^\perp$, then $(M-(a/|B|)K)w=0.$ Indeed, it follows from the facts that $a|A|=b|B|$, and $M\textbf{1}_A=b\textbf{1}_B=(a/|B|)K\textbf{1}_A$, $M\textbf{1}_B=a\textbf{1}_A=(a/|B|)K\textbf{1}_B$.
\end{itemize}

Since $e(X,Y)=\textbf{1}_Y^TM\textbf{1}_X$ and $|X||Y|=\textbf{1}_Y^TK\textbf{1}_X$, \[\left|e(X,Y)-\frac{a}{|B|}|X||Y|\right|=\left|\textbf{1}_Y^T(M-\frac{a}{|B|}K)\textbf{1}_X\right|.\]
For any vector $v$, let $\bar{v}$ be the orthogonal projection onto $W$, so that $\overline{v}\in W$, and $v-\overline{v}\in W^\perp$. Thus \[\textbf{1}_Y^T(M-\frac{a}{|B|}K)\textbf{1}_X=\textbf{1}_Y^T(M-\frac{a}{|B|}K)\overline{\textbf{1}_X}=\textbf{1}_Y^TM\overline{\textbf{1}_X}=\overline{\textbf{1}_Y}^T M \overline{\textbf{1}_X}, \;\text{so}\]
\[\left|e(X,Y)-\frac{a}{|B|}|X||Y|\right|\le \lambda_3||\overline{\textbf{1}_X}||~||\overline{\textbf{1}_Y}||.\]

Since \[ \overline{\textbf{1}_X}=\textbf{1}_X-\left((\textbf{1}_X\cdot \textbf{1}_A)/(\textbf{1}_A\cdot \textbf{1}_A)\right)\textbf{1}_A=\textbf{1}_X-(|X|/|A|)\textbf{1}_A,\] we have $||\overline{\textbf{1}_X}||=\sqrt{|X|(1-|X|/|A|)}$. Similarly, $||\overline{\textbf{1}_Y}||=\sqrt{|Y|(1-|Y|/|B|)}$.

In other words,
\[\left|e(X,Y)-\frac{a}{|B|}|X||Y|\right|\le \lambda_3\sqrt{|X||Y|(1-|X|/|A|)(1-|Y|/|B|)},\]
which completes the proof of the lemma.\end{proof}
\section{Proofs of Theorems \ref{dl2}, \ref{dl1}, and Corollary \ref{hello12}}

We start by proving the following lemma.

\begin{lemma}\label{bd1}
For any two $k$-tuples of vectors  $(\mathbf{a}_1,\ldots,\mathbf{a}_k)\ne (\mathbf{c}_1,\ldots,\mathbf{c}_k)$, we have $V_{\mathbf{a}_1,\ldots,\mathbf{a}_k}\ne V_{\mathbf{c}_1,\ldots,\mathbf{c}_k}$.
\end{lemma}
\begin{proof}
Since $(\mathbf{a}_1,\ldots,\mathbf{a}_k)\ne (\mathbf{c}_1,\ldots,\mathbf{c}_k)$, without loss of generality, we can assume that $\mathbf{a}_1\ne \mathbf{c}_1$.  Therefore,
\[f_1(\mathbf{x},\mathbf{a}_1)-f_1(\mathbf{x},\mathbf{c}_1)=(a_{11}-c_{11})x_1^{b_{11}}+\cdots+(a_{1d}-c_{1d})x_d^{b_{1d}}+a_{1(d+1)}-c_{1(d+1)},\]
 is a non-zero polynomial of degree at most $q-1$ in $\mathbb{F}_q[x_1,\ldots,x_d]$.

By Lemma~\ref{sch}, the cardinality of $V(f_1(\mathbf{x},\mathbf{a}_1)-f_1(\mathbf{x},\mathbf{c}_1))$ is at most $(q-1)q^{d-1}<q^d$. Let us observe that if $V_{\mathbf{a}_1,\ldots,\mathbf{a}_k}\equiv V_{\mathbf{c}_1,\ldots,\mathbf{c}_k}$, then $|V(f_1(\mathbf{x},\mathbf{a}_1)-f_1(\mathbf{x},\mathbf{c}_1))|=q^d$. This is indeed the case since each variety contains exactly $q^d$ points in $\mathbb{F}_q^d\times \mathbb{F}_q^k$. Thus, we obtain \[V\left(x_{d+1}-f_1(\mathbf{x},\mathbf{a}_1),\ldots, x_{d+k}-f_k(\mathbf{x},\mathbf{a}_k)\right)\ne V\left(x_{d+1}-f_1(\mathbf{x},\mathbf{c}_1),\ldots, x_{d+k}-f(\mathbf{x},\mathbf{c}_k)\right),\]
which completes the proof of the lemma.
\end{proof}

We define the bipartite graph $G=(A\cup B, E)$ as follows. The first vertex part $A$ is $(\mathbb{F}_q)^{d}\times (\mathbb{F}_q)^k$ and the second vertex part $B$ is the set of all varieties $V_{\mathbf{a_1},\ldots,\mathbf{a_k}}$ with $(\mathbf{a_1},\ldots,\mathbf{a_k})\in \left(\mathbb{F}_q^{d+1}\right)^k$. We draw an edge between a point $\mathbf{p} \in A$ and a variety $\mathbf{v}\in B$ if and only if $\mathbf{p}\in \mathbf{v}$. It is easy to check that $G$ is biregular with $\deg(A)=q^{dk}$ and $\deg(B) = q^d$.

\begin{lemma}\label{eigen1}
Let $\lambda_3$ be the third eigenvalue of the adjacency matrix of $G$. Then $|\lambda_3|\leq q^{dk/2}$.
\end{lemma}

\begin{proof}
Let $M$ be the adjacency matrix of $G$, so $M=\begin{bmatrix}
0&N\\
N^T&0
\end{bmatrix}$,
where $N$ is a $q^{d+k}\times q^{(d+1)k}$ matrix, with $N_{\mathbf{p}\mathbf{v}}=1$ if $\mathbf{p}\in \mathbf{v}$, $N_{\mathbf{p}\mathbf{v}}=0$ if $\mathbf{p}\not\in \mathbf{v}$.

Let $J$ be the $q^{d+k}\times q^{k(d+1)}$ all-one matrix and $K=\begin{bmatrix}
0&J\\
J^T&0
\end{bmatrix}.$
We prove that $M$ satisfies
\[M^3=q^{dk} M + (q^d-1)q^{k(d-1)}K.\]
 If $\mathbf{v}$ is an eigenvector corresponding to the third eigenvalue $\lambda_3$, then $K\mathbf{v}=0$. Therefore from the equation above one obtains that
$\lambda_3^3=q^{dk}\lambda_3,$ which implies that $|\lambda_3|= \sqrt{q^{dk}}$.

Let us observe that the $(\mathbf{p},\mathbf{v})$-entry of $M^3$ equals the number of walks of length three from $\p\in A$ to $\v\in B$,
that is the number of quadruples $(\p,\v',\p',\v)$, where $\p, \p'\in A, \v, \v'\in B,$ and $(\p, \v'), (\p', \v'),(\p', \v)$ are edges of $G$.

Given two points $\p=(p_1,\ldots,p_{d+k})$ and $\p'=(p_1',\ldots,p'_{d+k})$,
the varieties containing both $\p$ and $\p'$, and corresponding to $k$-tuples $(\mathbf{a_1},\ldots,\mathbf{a_k})\in \mathbb{F}_q^{(d+1)k}$ satisfies
\begin{equation}\label{eq:intersect111}
\begin{array}{ccl}
p_{d+i} & =&  h_i(p_1,\ldots,p_d)+ a_{i1}p_1^{b_{i1}}+\cdots+a_{id}p_d^{b_{id}}+a_{id+1}, \\
p_{d+i}' & =&   h_i(p_1',\ldots,p_d')+ a_{i1}(p_1')^{b_{i1}}+\cdots+a_{id}(p_d')^{b_{id}}+a_{id+1},
\end{array}
\end{equation}
for all $1\le i\le k$. Thus, for each $1\le i\le k$, we have
\begin{equation}\label{eq:subtract111}
p_{d+i}-p_{d+i}' = h_i(p_1,\ldots,p_d)-h_i(p_1',\ldots,p_d')+a_{i1}(p_1^{b_{i1}}-(p_1')^{b_{i1}})+\cdots+a_{id}(p_d^{b_{id}}-(p_d')^{b_{id}}).
\end{equation}
Let us observe that for each $a\in \mathbb{F}_q$, if $\gcd(r,q-1)=1$, then the equation $x^r=a^r$ has the unique solution $x=a$. Thus if $p_i=p_i'$ for all $1\le i \le d$, then there exists at least one variety containing both $p$ and $p'$ if and only if $p_{d+i}=p_{d+i}'$ for all $1\le i\le k$. This implies that $\p=\p'$.

We now count the number of walks of length three as follows. If $\p\not\in \v$, then we can choose $\p'\ne \p$ in $v$ such that $p_i\ne p_i'$ for some $1\le i\le d$ (otherwise, there is no $v'$ containing both $\p$ and $\p'$). We assume that $p_1\ne p_1'$, so $p_1^{b_{i1}}\ne (p_1')^{b_{i1}}$. Therefore, for each choice of $(a_{i2},\ldots, a_{id})$, $a_{i1}$ is determined uniquely by (\ref{eq:subtract111}), and $a_{id+1}$ is determined by any equation in (\ref{eq:intersect111}). In this case, the number of walks of length three is $(q^d-1)q^{k(d-1)}$.

If $\p\in \v$, then again there are $(q^d-1)q^{k(d-1)}$ walks with $\p\ne \p'$. Now we can choose $\p=\p'$. In this case, the number of walks equals the degree of $\p$. Thus if $\p\in \v$, then the number of walks of length three from $\p$ to $\v$ is $(q^d-1)q^{k(d-1)}+q^{dk}$.

In conclusion, $M$ satisfies
$M^3=q^{dk} M + (q^d-1)q^{k(d-1)}K$,
which completes the proof of the lemma.
\end{proof}
Combining Lemma \ref{expander} and Lemma \ref{eigen1}, Theorem \ref{dl2} follows.
We are now ready to prove Theorem \ref{dl1}.
\begin{proof}[Proof of Theorem \ref{dl1}]
 Let $\P'=\{p\times (0)^k: p\in\P\}$, then $|\P'|=|\P|$. Note that the number of incidences between points in $\P$ and varieties $W_{\mathbf{a}_1,\ldots,\mathbf{a}_k}$ is the number of incidences between points in $\P'$ and varieties $V_{\mathbf{a}_1,\ldots,\mathbf{a}_k}$. Therefore, Theorem \ref{dl1} follows immediately from Theorem \ref{dl2}.
\end{proof}
\begin{proof}[Proof of Corollary \ref{hello12}]
Let $s$ be the sphere of radius $r$ with the center $\mathbf{a}\in \F^d$, that is the set of points $(x_1,\cdots,x_d)\in\F^d$ satisfying
$(x_1-a_1)^2+\cdots+(x_d-a_d)^2=r.$
Therefore, we can re-write the formula for the points contained in $s$ as
\[x_1^2+\cdots x_d^2+\sum_{i=1}^da_ix_i-(r-\sum_{i=1}^da_i^2)=0.\]
Let $h(\mathbf{x})=x_1^2+\cdots+x_d^2$, $\mathbf{b}=(1, \ldots, 1)$, and $\mathbf{a}=\left(a_1,\ldots, a_d, -\left(r-\sum_{i=1}^da_i^2\right)\right)$. Then Corollary \ref{hello12} follows immediately from Theorem \ref{hello11}.
\end{proof}

\section{Proof of Theorem \ref{hello2}}
Let us define
\[\S:=\left\lbrace x_{d+1}=F(\mathbf{x},\q)\colon \q\in \P\right\rbrace, ~\P':=\left\lbrace(\p,t)\in\mathbb{F}_q^{d+1}\colon (\p,t)\in \P\times \Delta_{F}(\P, \p)\right\rbrace.\]
Since $F(\mathbf{x},\mathbf{y})$ is non-degenerate, $\S$ is a set of hypersurfaces. It follows from Theorem \ref{dl2} for the case $k=1$ that
\[e(\P',\S)\le \frac{|\P'||\S|}{q}+q^{d/2}\sqrt{|\P'||\S|}.\]
On the other hand, we have $e(\P',\S)=|\P|^2$, thus
\begin{eqnarray}\label{18}
|\P|^2 \le e(\P',\S)&\le& \frac{|\P'||\S|}{q}+q^{d/2}\sqrt{|\P'||\S|}.\nonumber\\
&=&\frac{|\P|\sum_{\p\in \P}|\Delta_F(\P,\p)|}{q}+q^{d/2}\sqrt{|\P|\sum_{\p\in \P}|\Delta_F(\P, \p)|}.
\end{eqnarray}
If $\sum_{\p\in \P}|\Delta_F(\P, \p)|\le (1-c^2)q|\P|$, then it follows from (\ref{18}) that \[ |\P|^2\le |\P|^2(1-c^2)+q^{(d+1)/2}|\P|\sqrt{(1-c^2)}.\]
This implies that
\begin{align*}
|\P|&<\sqrt{\frac{(1-c^2)}{c^4}}q^{(d+1)/2},
\end{align*}
which is a contradiction. Therefore, \begin{equation}\label{ineq1}\frac{1}{|\P|}\sum_{\p\in \P}|\Delta_F(\P, \p)|> (1-c^2)q.\end{equation}
Let $\P':=\lbrace \p\in \P\colon |\Delta_F(\P, \p)|>(1-c)q\rbrace$. Suppose that $|\P'|<(1-c)|\P|$, we have

\[\sum_{\p\in \P\setminus \P'}|\Delta_F(\P, \p)|\le (|\P|-|\P'|)(1-c)q, \text{ and } \sum_{\p\in \P'}|\Delta_F(\P, \p)|\le q|\P'|.\]
Putting everything together, we obtain
$$\sum_{\p\in \P}|\Delta_F(\P, \p)|\le  (1-c)q|\P|+cq|\P'|<(1-c)q|\P|+cq(1-c)|\P|=(1-c^2)q|\P|,$$
which contradicts (\ref{ineq1}), and the theorem follows.
\section{Proofs of Theorem \ref{beckthm} and Theorem \ref{radii}}
We first need the following lemma.
\begin{lemma}
There is a unique sphere in $\mathbb{F}_q^3$ passing through four given non-coplanar points.
\end{lemma}
\begin{proof}
Let $\mathbf{p}_1=(a_1, a_2, a_3)$, $\mathbf{p}_2=(b_1, b_2, b_3)$, $\mathbf{p}_3=(c_1, c_2, c_3)$, and $\mathbf{p}_4=(d_1, d_2, d_3)$ be given non-coplanar points. We will show that there exists a unique sphere in $\mathbb{F}_q^3$  containing $\mathbf{p}_1$, $\mathbf{p}_2$, $\mathbf{p}_3$, and $\mathbf{p}_4$. In fact, a sphere passing through these four points can be written as
\[(x-e_1)^2+(y-e_2)^2+(z-e_3)^2=r, ~\textup{with}~ e_1, e_2, e_3, r\in \mathbb{F}_q.\]
Let $e_1'=-2e_1, e_2'=-2e_2, e_3'=-2e_3$, and $r'=e_1^2+e_2^2+e_3^2-r$. Then we obtain the following system of four equations
\begin{eqnarray*}\label{ssf}
a_1e_1'+a_2e_2'+a_3e_3'+r'&=&-a_1^2-a_2^2-a_3^2\\
b_1e_1'+b_2e_2'+b_3e_3'+r'&=&-b_1^2-b_2^2-b_3^2\\
c_1e_1'+c_2e_2'+c_3e_3'+r'&=&-c_1^2-c_2^2-c_3^2\\
d_1e_1'+d_2e_2'+d_3e_3'+r'&=&-d_1^2-d_2^2-d_3^2
\end{eqnarray*}
This system can be written as
\begin{equation}\label{eqf}\begin{pmatrix}
a_1&a_2&a_3&1\\
b_1&b_2&b_3&1\\
c_1&c_2&c_3&1\\
d_1&d_2&d_3&1
\end{pmatrix}\begin{pmatrix}
e_1'\\e_2'\\e_3'\\r'
\end{pmatrix}=\begin{pmatrix}
-a_1^2-a_2^2-a_3^2\\-b_1^2-b_2^2-b_3^2\\-c_1^2-c_2^2-c_3^2\\-d_1^2-d_2^2-d_3^2
\end{pmatrix}\end{equation}
Since $\mathbf{p}_i$'s are non-coplannar points, the determinant of the matrix on the left hand side of (\ref{eqf}) is not equal to $0$. Therefore, the system \ref{ssf} has an unique solution. In short, there is a unique sphere passing through any four given non-coplanar points.
\end{proof}
\begin{proof}[Proof of Theorem \ref{beckthm}]
Since $|\P|\ge 8q^2$, by the pigeon-hole principle, there exist two parallel planes $U$ and $V$ satisfying $|U\cap \P|\ge 5q$ and $|V\cap \P|\ge 8q$.  Let $\gamma$ be the direction which is orthogonal to $U$ and $V$. We set $E_1:=U\cap \P$ and $E_2:=V\cap \P$. It follows from Theorem \ref{becktype} that there are at least $4q^3/9$ distinct circles in $U$ determined by $E_1$. We denote the set of centers of these circles by $F_1$.

Let $f$ be the projection from $U$ to $V$ in the direction $\gamma$, and $F_2:=f(F_1)$. Then we have $|F_1|=|F_2|\ge 4q^2/9$. It follows from Corollary \ref{bodeds} that there exists a set $F_2'\subseteq F_2$ such that, for each point $\mathbf{p}\in F_2'$, we have $|\Delta_{\mathbb{F}_q}(F_2, \mathbf{p})|\ge q/2$. Thus, for each point $\mathbf{p}\in F_2'$, there exist at least $q/2$ circles centered at $\mathbf{p}$ of radii in $\Delta_{\mathbb{F}_q}(F_2, \mathbf{p})$. We denote the set of these circles by $C_{\mathbf{p}}$.

We note that $|F_2\setminus F_2'|=o(q^2)$, so the number of circles in $U$ with centers in $F_1\setminus f^{-1}(F_2')$ is $o(q^3)$. Therefore, the number of circles in $U$ with centers in $f^{-1}(F_2')$ is at least $2q^3/9$.

On the other hand, for each point $\mathbf{p}\in F_2'$, the spheres determined by a circle center at $f^{-1}(\mathbf{p})\in U$ and circles in $C_{\mathbf{p}}$ are distinct. Let $S_{\mathbf{p}}$ denote the set of the spheres associating $\mathbf{p}\in F_2'$. Then $S_{\mathbf{p}}\cap S_{\mathbf{q}}=\emptyset$ for any two points $\mathbf{p}$ and $\mathbf{q}$ in $F_2'$. Hence, the number of distinct spheres determined by $\P$ is at least $q^4/9$, and the theorem follows.
\end{proof}
\begin{proof}[Proof of Theorem \ref{radii}]
It follows from Theorem \ref{beckthm} that if $|\P|\ge 8q^2$, then the number of spheres determined by $\P$ is at least $q^4/9$. Since the cardinality of the set of centers of these spheres is at most $q^3$, the number of distinct radii of spheres determined by $\P$ is at least $q/9$, and the theorem follows.
\end{proof}
\section{Distinct distances between points and lines}
To prove results on the number of distinct distances between points and lines, we construct the \textit{point-line distance bipartite graph} as follows.
\subsection{Point-line distance bipartite graph}
Let $SQ:=\{x^2\colon x\in \F\}\setminus\{0\}$. We define the point-line distance bipartite graph $PL(\mathbb{F}_q^2)=(A\cup B, E)$ as follows. The first vertex part, $A$, is the set of all quadruples $(a, b, c, \lambda)\in \F^4$ satisfying $(a^2+b^2)\lambda\in SQ$. The second vertex part, $B$, is the set of all points in $\F^2$. There is an edge between $(a,b,c,\lambda)$ and $(x,y)$ if and only if $(ax+by+c)^2=\lambda(a^2+b^2)$. We have the following properties of the point-line distance bipartite graph $PL(\F^2)$.
\begin{lemma}\label{hai}
The degree of each vertex in $A$ is $2q$, and the degree of each vertex in $B$ is $2|S|$, where
\[S=\left\lbrace (a, b, \lambda)\in \F^3\colon \lambda(a^2+b^2)\in SQ\right\rbrace.\]
\end{lemma}
\begin{proof}
Let $(a, b, c, \lambda)$ be a vertex in $A$. The degree of $(a, b, c, \lambda)$ is the number of solutions $(x, y)\in \F^2$ of the equation
\begin{equation}\label{mot}
(ax+by+c)^2=\lambda(a^2+b^2).
\end{equation}
Since $\lambda(a^2+b^2)\in SQ$, there exists $m\in \F\setminus \{0\}$ such that $\lambda(a^2+b^2)=m^2$. Since $(a,b,c,\lambda)$ is fixed, it follows from the equation (\ref{mot}) that $(x,y)$ is a solution of either equations of the following system
\[ax+by+c=m,~ax+by+c=-m.\]
Since $(a^2+b^2)\ne 0$, we can assume that $a\ne 0$. Therefore, for any choice of $y$ from $\F$, $x$ is determined uniquely, so the degree of $(a, b, c, \lambda)$ is $2q$.

Let $(x,y)$ be a vertex in $B$. The degree of $(x,y)$ is the number of solutions $(a, b, c,\lambda)\in A$ satisfying the equation (\ref{mot}). Note that if there is an edge between $(x,y)$ and $(a,b,c,\lambda)$, then $\lambda(a^2+b^2)\in SQ$ (this follows from the definition of $A$). Thus, for each triple $(a,b,\lambda)\in S$, we assume that $\lambda(a^2+b^2)=m^2$ for some $m\in \F$. It follows from
the equation (\ref{mot}) that $c$ is a solution of either equations of the following system
\[ax+by+c=m, ~ax+by+c=-m.\]
This implies that, for each triple $(a,b,\lambda)\in S$, there are exactly two values of $c$ such that $(a,b,c,\lambda)$ is adjacent to $(x,y)$. In short, the degree of each vertex in $B$ is $2|S|$.\end{proof}

\begin{lemma}\label{one}
Let $(a,b,c,\lambda)$ and $(d,e, f,\beta)$ be two distinct vertices in $A$, and $N$ be the number of  common neighbors of $(a,b,c,\lambda)$ and $(d,e,f,\beta)$. Then we have
\[N=\begin{cases}
q, ~\mbox{if $(d,e)=k(a,b)$ and $f\ne kc$, for some $k\in \F\backslash\{0\}$}\\
0, ~\mbox{if $(d,e,f)=k(a,b,c)$ and $\lambda\ne \beta$, for some $k\in \F\backslash\{0\}$}\\
2q, ~\mbox{if $(d,e,f)=k(a,b,c)$ and $\lambda= \beta$, for some $k\in \F\backslash\{0\}$}\\
4, ~\mbox{otherwise}
\end{cases}
\]
\end{lemma}
\begin{proof}
The number of common neighbors of $(a,b,c,\lambda)$ and $(d,e,f,\beta)$ is the number of solutions $(x,y)\in \F^2$ of the following system
\begin{equation}
\begin{array}{ccl}\label{haha}
(ax+by+c)^2&=&\lambda(a^2+b^2)=m_1^2\\
(dx+ey+f)^2&=&\beta(d^2+e^2)=m_2^2,
\end{array}\end{equation}
for some $m_1, m_2\in \F\setminus\{0\}$.
This implies that $(x,y)$ is a solution of one of the following $4$ systems formed of the following $2$ equations, each system corresponding to a choice of $\pm$.
\[ax+by+c=\pm m_1, ~dx+ey+f=\pm m_2\]
Since $m_1, m_2\in \F\setminus\{0\}$, two such systems do not have a common solution.
We have the two following cases:
\begin{enumerate}
\item If $(a,b)$ and $(d,e)$ are linearly independent, then each system has unique solution, so the number of common neighbors of $(a,b,c,\lambda)$ and $(d,e,f, \beta)$ is $4$.
\item If $(a,b)$ and $(d,e)$ are linearly dependent, then we assume that $(d,e)=k(a,b)$ for some $k\in \F\backslash\{0\}$. It follows from the system (\ref{haha}) that
\begin{eqnarray}\label{haha1}
k^2(ax+by+c)^2&=&k^2\lambda(a^2+b^2)\nonumber\\
(kax+kby+f)^2&=&k^2\beta(a^2+b^2)\nonumber
\end{eqnarray}
Subtracting the first equation from the second equation, we obtain
\begin{equation}\label{ee}
(f-kc)(2kax+2kby+f+kc)=(a^2+b^2)k^2(\lambda-\beta).
\end{equation}

If $f=kc$ and $\lambda= \beta$, then the number of solutions $(x,y)$ of the system (\ref{haha}) is $\deg(a,b,c,\lambda)$, which by Lemma \ref{hai} equals $2q$.

If $f=kc$ and $\lambda\ne \beta$, then the number of solutions $(x,y)$ of the system (\ref{haha}) is $0$.

If $f\ne kc$, then from equation (\ref{ee}) follows that
\begin{equation}\label{aa}
2kax+2kby+f+kc=(f-kc)^{-1}(a^2+b^2)k^2(\lambda-\beta).
\end{equation}
Since $a^2+b^2\ne 0$, we assume that $a\ne 0$. Therefore, the number of solutions of the equation (\ref{aa}) is $q$, since we can choose $y$ arbitrary, and for each choice of $y$, $x$ is determined uniquely by the equation (\ref{aa}).  In other words, in this case, the number of common neighbors of $(a,b,c,\lambda)$ and $(d,e,f,\beta)$ is $q$.
\end{enumerate}
\end{proof}
In the following two lemmas, we count the number of walks of length three between a vertex $(a,b,c,\lambda)$ from $A$ and a vertex $(z,t)$ from $B$. This will be directly
related to obtaining the value for the third eigenvalue corresponding to the point-line distance graph. The first lemma treats the case when $(a,b,c,\lambda)$ and $(z,t)$ are
not adjacent, while the second lemma deals with the case when the two vertices are adjacent.

\begin{lemma}\label{nonajd}
Given a pair of non-adjacent vertices $(a,b,c,\lambda)$ and $(z,t)$, let $N$ be the number of walks of length three between them. Then we have
\[N=
\begin{cases}
4\left( 2|S|-(q-1)^2\right)+q\left( q-1\right)^2, ~\textup{if}~ az+bt+c=0\\
4\left( 2|S|-(q-1)^2\right)+q\left((q-1)^2-(q-1)\right), ~\textup{otherwise}
\end{cases}\]
\end{lemma}
\begin{proof}
We can distinguish two cases:
\begin{enumerate}
\item The point $(z,t)$ lies on the line $ax+by+c=0$.
\begin{enumerate}
\item  First we count the number of neighbors $(d,e,f,\beta)$ of $(z,t)$ satisfying $(d,e)\ne k(a,b)$ for all $k\in \F\backslash\{0\}$.  It follows from the definition of $S$ that the number of triples $(d,e,\beta)$ satisfying $\beta(d^2+e^2)\in SQ$ and $d^2+e^2\ne k^2(a^2+b^2)$ for all $k\in \F\backslash\{0\}$ is $|S|-(q-1)^2/2$. Moreover, with each triple $(d,e,\beta)$ satisfying $\beta(d^2+e^2)\in SQ$, there are two solutions of $f$ such that $(d,e,f,\beta)$ is a neighbor of $(z,t)$. Thus, the number of neighbors $(d,e,f,\beta)$ of $(z,t)$ satisfying $(d,e)\ne k(a,b)$ for all $k\in \F\backslash\{0\}$ is $2|S|-(q-1)^2$.
\item We now count the number of neighbors $(d,e,f,\beta)$ of $(z,t)$ satisfying $(d,e,f)=k(a,b,c)$, for some  $k\in \F\backslash\{0\}$, and $\lambda\ne \beta$. If $(d,e,f,\beta)$ is a neighbor of $(z,t)$, then $(dz+et+f)^2=\beta(d^2+e^2)$, which implies that $k^2(az+bt+c)^2=\beta k^2(a^2+b^2)$.
Since $(z,t)$ lies on the line $ax+by+c=0$, we have $\beta=0$. Thus there is no neighbor $(d,e,f,\beta)$ of $(z,t)$ satisfying $(d,e,f)=k(a,b,c)$, and $\lambda\ne \beta$ for some $k\in \F\backslash\{0\}$.
\item Combining two above cases, we obtain that the number of neighbors $(d,e,f,\beta)$ of $(z,t)$ satisfying $(d,e)= k(a,b)$, $f\ne kc$ for some $k\in \F\backslash\{0\}$ is $(q-1)^2$.
\end{enumerate}
Thus, if $az+bt+c=0$, the number of walks of length three between $(z,t)$ and $(a,b,c,\lambda)$ is
$q\left(q-1\right)^2+4\left(2|S|-(q-1)^2\right)$, which finishes this case.
\item The point $(z,t)$ does not lie on $ax+by+c=0$.
\begin{enumerate}
\item By the same arguments as above, the number of neighbors $(d,e,f,\beta)$ of $(z,t)$ satisfying $(d,e)\ne k(a,b)$ for all $k\in \F\backslash\{0\}$ is $2|S|-(q-1)^2$.
\item The number of neighbors $(d,e,f,\beta)$ of $(z,t)$ satisfying $(d,e,f)=k(a,b,c), k\in \F\backslash\{0\}$, and $\lambda\ne \beta$ is $(q-1)$. Indeed, if $(d,e,f,\beta)$ is a neighbor of $(z,t)$, then $(dz+et+f)^2=\beta(d^2+e^2)$, which implies that $k^2(az+bt+c)^2=\beta k^2(a^2+b^2)$.
Since $(z,t)$ does not  lie on the line $ax+by+c=0$, $\beta=(az+bt+c)^{-2}(a^2+b^2)\ne 0$. It is easy to see that $\beta\ne \lambda$. Since there are $q-1$ choices of $k$, the number of neighbors $(d,e,f,\beta)$ of $(z,t)$ satisfying $(d,e,f)=k(a,b,c), k\in \F\backslash\{0\}$, and $\lambda\ne \beta$ is $(q-1)$.
\item Combining above cases implies that the number of neighbors $(d,e,f,\beta)$ of $(z,t)$ satisfying $(d,e)= k(a,b)$, $f\ne kc$ for some $k\in \F\backslash\{0\}$ is $(q-1)^2-(q-1)$.
\end{enumerate}
In other words, if $az+bt+c\ne 0$, then the number of walks of length three between $(z,t)$ and $(a,b,c,\lambda)$ is
$q\left((q-1)^2-(q-1)\right)+4\left(2|S|-(q-1)^2\right)$.
\end{enumerate}\end{proof}
%
%

\begin{lemma}\label{adjacent}
Given a pair of adjacent vertices $(a,b,c,\lambda)$ and $(z,t)$, the the number of walks of length three between them is $4\left(2|S|-(q-1)^2\right)+2q(q-1)+q\left((q-1)^2-(q-1)\right)$.
\end{lemma}
\begin{proof}
We now consider the following cases:
\begin{enumerate}
\item  As in the proof of Lemma \ref{nonajd}, we obtain that the number of neighbors $(d,e,f,\beta)$ of $(z,t)$ satisfying $(d,e)\ne k(a,b)$ for all $k\in \F\backslash\{0\}$ is $2|S|-(q-1)^2$.
\item  The number of neighbors $(d,e,f,\beta)$ of $(z,t)$ satisfying $(d,e,f)= k(a,b,c)$ and $\lambda=\beta$ (for some $k\in \F\backslash\{0\}$) is $(q-1)$, since $(a,b,c,\lambda)$ is a neighbor of $(z,t)$.
\item The number of neighbors $(d,e,f,\beta)$ of $(z,t)$ satisfying $(d,e,f)= k(a,b,c)$, and $\lambda\ne\beta$, for some $k\in \F\backslash\{0\}$ is $0$ since $(a,b,c,\lambda)$ is a neighbor of $(z,t)$.
\item Combining above cases implies that the number of neighbors $(d,e,f,\beta)$ of $(z,t)$ satisfying $(d,e)= k(a,b)$, $f\ne kc$ for some $k\in \F\backslash\{0\}$ is $(q-1)^2-(q-1)$.
\end{enumerate}
In other words, this gives that the number of walks of length three from $(a,b,c,\lambda)$ to $(z,t)$ is $q((q-1)^2-(q-1))+2q(q-1)+4\left(2|S|-(q-1)^2\right)$, which completes the proof.\end{proof}

\begin{theorem}\label{hello}
The absolute value of the third eigenvalue of the point-line distance graph $PL(\F^2)$ is at most $2q^{4/3}$.
\end{theorem}
\begin{proof}
Let $M$ be the adjacency matrix of $PL(\F^2)$, which has the form
$$M=\begin{bmatrix}
0&N\\
N^T&0
\end{bmatrix},
$$
where $N$ is a $|A|\times |B|$ matrix, and $N_{(a,b,c,\lambda), (x,y)}=1$ if there is an edge between $(a,b,c,\lambda)\in A$ and $(x,y)$ in $B$, and zero otherwise. Therefore,
$$M^3=\begin{bmatrix}
0&NN^TN\\
N^TNN^T&0
\end{bmatrix}.
$$
Let $J$ be the $|A|\times |B|$ all-one matrix. We set
\[K=\begin{bmatrix}
0&J\\
J^T&0
\end{bmatrix}.\]
It follows from Lemma \ref{nonajd} and Lemma \ref{adjacent} that
\begin{equation}\label{eqaa}M^3=\left(4\left(2|S|-(q-1)^2\right)+q((q-1)^2-(q-1))\right)K+2q(q-1)M+q(q-1)A_\mathcal{IN},\end{equation}
where $A_\mathcal{IN}$ is the adjacency matrix of the bipartite graph $\mathcal{IN}=(A\cup B,E_{\mathcal{IN}})$ defined as follows: there is an edge between $(z,t)\in B$ and $(a,b,c,\lambda)\in A$ if and only if $(z,t)$ lies on the line $ax+by+c=0$.  It is easy to check that in the graph $\mathcal{IN}$, the degree of each vertex $(z,t)$ is $|S|$, and the degree of each vertex $(a,b,c,\lambda)$ is $q$. Thus, the largest eigenvalue of $\mathcal{IN}$ is bounded from above by $\sqrt{q|S|}$.

Let $\mathbf{v_3}$ be an eigenvector corresponding to the third eigenvalue of the point-line distance graph $PL(\F^2)$. Then it follows from the equation (\ref{eqaa}) that
\[(\lambda_3^3-2q(q-1)\lambda_3)\mathbf{v_3}=q(q-1)A_{\mathcal{IN}}\mathbf{v_3},\]
since $K\mathbf{v_3}=0$. This implies that $\mathbf{v_3}$ is an eigenvector of the matrix $q(q-1)A_{\mathcal{IN}}$, with the corresponding eigenvalue
\[\lambda_3^3-2q(q-1)\lambda_3\le q(q-1)\sqrt{q|S|}.\]
Note that if $-1$ is not a square in $\F$, then $|S|=(q-1)^2(q+1)/2$, and if $-1$ is a square in $\F$, then $|S|=(q-1)(q^2-2q+1)/2$. Thus, in both cases, we have $|S|\le q^3$.  This implies that
\[\lambda_3^3-2q(q-1)\lambda_3\le q^4,\]
Let $f(x)=x^3-2q(q-1)x-q^4=0$, so $f'(x)=3x^2-2q(q-1)$. Thus, $f'(x)\ge 0$ if $x\ge \sqrt{2q(q-1)/3}$. On the other hand, if $x=2q^{4/3}$, then $f(x)> 0$, which implies that $f(x)\le 0$ if $x<2q^{4/3}$. Therefore, $\lambda_3\le 2q^{4/3}$, which concludes the proof of the theorem.\end{proof}

\subsection{Proofs of Theorem \ref{thm1} and Corollary \ref{co1}}
In order to prove Theorem \ref{thm1}, we need the following result, proved in \cite{vinh-incidence}, which is also a corollary of Theorem \ref{dl2}.
\begin{theorem}[\textbf{Vinh}, \cite{vinh-incidence}]\label{pointlinein}
Let $\P$ be a set of points and $\L$ a set of lines. Then
\[\left\vert I(\P, \L)-\frac{|\P||\L|}{q}\right\vert\le q^{1/2}\sqrt{|\P||\L|},\]
where $I(\P, \L)$ represents the number of incidences between points in $\P$ and lines in $\L$.
\end{theorem}
\begin{proof}[Proof of Theorem \ref{thm1}]
First we prove that \[\frac{1}{|\L|}\sum_{l\in \L}|\Delta_{\mathbb{F}_q}(l,\P)|> (1-c^2)q.\]
For each $l\in \L$, let denote the set of non-zero distances between $l$ and $\P$ by $\Delta_{\F}(l, \P)$. For each line $l=\{ax+by+c=0\}$, we define
\[D_l:=\{(a,b,c, \lambda)\colon  \lambda\in \Delta_{\F}(l, \P) \}.\]
Let $D=\cup_{l\in \L}D_l$.  Since $\L$ is a set of lines in $\F^2$, $D$ becomes a set of points in $\F^4$. Therefore, $|D_l|=|\Delta_{\F}(l, \P)|$, and $|D|=\sum_{l\in \L}|\Delta_{\F}(l, \P)|$. One can observe that each point $(a,b,c,\lambda)$ in $D$ satisfies the condition $\lambda(a^2+b^2)\in SQ$. It follows from the definition of $D$ and Theorem \ref{pointlinein} that
\begin{equation}\label{eq2}e(D, \P)=|\P||\L|-I(\P, \L)\ge \frac{|\P||\L|}{2},\end{equation}
where $e(D, \P)$ is the number of edges between $D$ and $\P$ in the point-line graph. On the other hand, we now prove that
\[e(D, \P)\le 2(1-c^2)|\P||\L|+q^{4/3}\sqrt{(1-c^2)|\P||\L|}.\]
Let $U=\left\lbrace (ka,kb,kc,\lambda)\colon k\in \F\backslash\{0\}, (a,b,c,\lambda)\in D \right\rbrace$.
Since the lines in $\L$ are distinct, no two points from $U$ coincide, so $|U|=(q-1)|D|$. If there is an edge between a point $p\in \P$ and a point $(a,b,c, \lambda)$ in $D$, then there is also an edge between $p$ and each point $(ka,kb,kc,\lambda)\in U$ where $k\in \F\backslash\{0\}$. Therefore, $e(U, \P)=(q-1)e(D, \P)$. On the other hand, it follows from Lemma \ref{expander} and Lemma \ref{hello} that
\begin{eqnarray}
e(U, \P)\le \frac{2|U||\P|}{q}+2q^{4/3}\sqrt{|U||\P|}&=&\frac{2|\P|(q-1)\sum_{l\in \L}|\Delta_{\F}(l, \P)|}{q}\nonumber\\&+&2q^{4/3}\sqrt{|\P|(q-1)\sum_{l\in \L}|\Delta_{\F}(l, \P)|}\nonumber.
\end{eqnarray}
If $\sum_{l\in \L}|\Delta_{\mathbb{F}_q}(l,\P)|< (1-c^2)q|\L|$ with $2(1-c^2)<1/2$, then we obtain
\begin{eqnarray}
e(U, \P)&\le& \frac{2|\P|(q-1)\sum_{l\in \L}|\Delta_{\F}(l, \P)|}{q}+2q^{4/3}\sqrt{|\P|(q-1)\sum_{l\in \L}|\Delta_{\F}(l, \P)|}\nonumber\\
&<&2(1-c^2)|\P|(q-1)|\L|+2q^{4/3+1}\sqrt{(1-c^2)|\P||\L|}.
\end{eqnarray}
Since $e(U, \P)=(q-1)e(D, \P)$, we obtain
\begin{equation}\label{eq1}e(D, \P)\le 2(1-c^2)|\P||\L|+2q^{4/3}\sqrt{(1-c^2)|\P||\L|}.\end{equation}
Combining the equation (\ref{eq2}) and the equation (\ref{eq1}), we obtain
\[\left(1/2-2(1-c^2)\right)|\P||\L|\le 2\sqrt{1-c^2}q^{4/3}\sqrt{|\P||\L|},\]
which implies that
\[|\P||\L|\le \frac{4(1-c^2)}{(1/2-2(1-c^2))^2}q^{8/3},\]
which contradicts the assumption in the hypothesis. In other words, we have that
\begin{equation}\label{hehehe}\frac{1}{|\L|}\sum_{l\in \L}|\Delta_{\mathbb{F}_q}(l,\P)|> (1-c^2)q.\end{equation}
Let $\L':=\lbrace l\in \L: |\Delta_{\F}(\P,l)|>(1-c)q\rbrace$. Suppose  that $|\L'|<(1-c)|\L|$, so
\begin{equation}\label{pt11.3}
\sum_{l\in \L\setminus \L'}|\Delta_{\F}(\P,l)|\le (|\L|-|\L'|)(1-c)q, \; \text{and}
\end{equation}
 \begin{equation}\label{pt11.2} \sum_{l\in \L'}|\Delta_{\F}(\P,l)|\le q|\L'|.\end{equation}
Putting (\ref{pt11.3}) and (\ref{pt11.2}) together, we obtain
$$\sum_{l\in \L}|\Delta_{\F}(\P,l)|\le  (1-c)q|\L|+cq|\L'|<(1-c)q|\L|+cq(1-c)|\L|=(1-c^2)q|\L|,$$
which contradicts (\ref{hehehe}). Therefore, there exists a subset $\L'$ of $\L$ such that $|\L'|=(1-o(1))|\L|$, and $|\Delta_{\F}(\P, l)|\gtrsim q$, for each $l\in \L'$, which completes the proof.\end{proof}

In order to prove Corollary \ref{co1}, we need the following result, which as already mentioned in the introduction, is a variant of Beck's theorem over finite fields.
\begin{theorem}(\cite[Corollary 5]{lund})\label{thm4}
Let $\P$ be a set of points in $\F^2$. If $|\P|\ge 3q$, then the number of distinct lines determined by $\P$ is at least $q^2/3$.
\end{theorem}
\begin{proof}[Proof of Corollary \ref{co1}]
One can check that the number of degenerate lines is at most $2q$. Therefore, the proof of Corollary \ref{co1} follows immediately from Theorem \ref{thm1} and Theorem \ref{thm4}.\end{proof}

\end{document}